\documentclass[12pt]{amsart}
\usepackage{latexsym,amscd,amssymb,epsfig} 
\usepackage[dvips]{color}
\usepackage[all]{xy}
\usepackage[T1]{CJKutf8}

\theoremstyle{plain}
  \newtheorem{theorem}{Theorem}[section]
  \newtheorem{proposition}[theorem]{Proposition}
  
  \newtheorem{corollary}[theorem]{Corollary}
  
\theoremstyle{definition}
  \newtheorem{definition}[theorem]{Definition}
  \newtheorem{example}[theorem]{Example}

\theoremstyle{remark}
  \newtheorem{remark}[theorem]{Remark}

\numberwithin{equation}{section}

\def\umapright#1{\smash{
   \mathop{\longrightarrow}\limits^{#1}}}

\def\rmapdown#1{\Big\downarrow\rlap
   {$\vcenter{\hbox{$\scriptstyle#1$}}$}}

\def\tempbaselines
{\baselineskip22pt\lineskip3pt
   \lineskiplimit3pt}
\def\diagram#1{\null\,\vcenter{\tempbaselines
\mathsurround=0pt
    \ialign{\hfil$##$\hfil&&\quad\hfil$##$\hfil\crcr
      \mathstrut\crcr\noalign{\kern-\baselineskip}
  #1\crcr\mathstrut\crcr\noalign{\kern-\baselineskip}}}\,}

\def\pullback#1&#2&#3&#4&#5&#6&#7&#8&{
\diagram{#1&\umapright{#2}&#3\cr
\rmapdown{#4}&&\rmapdown{#5}\cr
#6&\umapright{#7}&#8\cr}}

\def\calA{{\mathcal A}}
\def\calB{{\mathcal B}}
\def\calC{{\mathcal C}}
\def\calD{{\mathcal D}}

\def\calF{{\mathcal F}}
\def\calK{{\mathcal K}}

\def\calO{{\mathcal O}}
\def\calP{{\mathcal P}}
\def\calQ{{\mathcal Q}}

\def\calS{{\mathcal S}}

\def\frakM{{\mathfrak M}}
\def\frakN{{\mathfrak N}}
\def\k{{\underline{k}}}

\def \Aut{\mathop{\rm Aut}\nolimits}

\def\colim{\mathop{\varprojlim}\nolimits}
\def\CM{\mathop{\underline{\rm CM}}\nolimits}

\def\H{\mathop{\rm H}\nolimits}
\def\I{\mathop{\mathbb I}\nolimits}

\def \Hom{\mathop{\rm Hom}\nolimits} 
\def \hotimes{\mathop{\hat{\otimes}}\nolimits} 
 
\def \Id{\mathop{\rm Id}\nolimits}

\def\L{{\mathbb{L}{\rm K}}}
\def\lim{\mathop{\varinjlim}\nolimits}
 
\def \Ob{\mathop{\rm Ob}\nolimits} 
\def \Mor{\mathop{\rm Mor}\nolimits}
\def\R{{\mathbb{R}{\rm es}}} 
\def\r{\underline{R}}

\def\Res{\mathop{\rm Res}\nolimits}

\def\sd{\mathop{\rm sd}\nolimits}
\def\sfC{\mathop{\sf C}\nolimits}
\def\sfP{\mathop{\sf P}\nolimits}
\def\st{\mathop{\rm St}\nolimits}

\def\Stab{\mathop{\rm Stab}\nolimits}

\def\T{\mathop{\mathbb T}\nolimits}

\def\SG{{S\rtimes G}}

\def\PG{{\calP\rtimes G}}

\def\QH{{\calQ\rtimes H}}

\def\CC{{\Bbb C}}

\begin{document}

\title[Generalized inductions]
{On homology with coefficients and generalized inductions}

\author{Fei Xu}
\email{fxu@stu.edu.cn}
\address{Department of Mathematics\\
Shantou University\\
Shantou, Guangdong 515063, China}

\subjclass[2010]{}


\thanks{The author \begin{CJK*}{UTF8}{}
\CJKtilde \CJKfamily{gbsn}(徐 斐)
\end{CJK*} is supported by the NSFC grant No. 11671245}

\begin{abstract}
In group representations several inductions given by tensoring with appropriate bimodules may be reconstructed via homology of $G$-posets with $G$-equivariant coefficients. For this purpose, we need various local categories of a finite group $G$, which afford the coefficients. Consequently, the functors among local categories give rise to the homology constructions naturally, and may be used to reformulate some existing results, as well as to deduce new statements.
\end{abstract}

\maketitle

\section{Introduction}

Deligne and Lusztig \cite[Introduction]{DeLu} constructed their induced representations ``...as the alternating sum of the cohomology with compact support of (some variety), with coefficients in certain $G$-equivariant locally constant $l$-adic sheaves of rank one''. Here $G$ is a finite group of Lie type. This idea has been immediately taken by various authors, see for instance \cite{CuLe, Sm}, for representations over fields of characteristic zero or $p$. An early account of the (co)homology theory of $G$-spaces with coefficients in $G$-equivariant sheaves can be found in \cite{De}. Undoubtedly, (co)homology with $G$-equivariant coefficients is also heavily studied in representations and (co)homology of \textit{abstract} finite groups, see \cite{Be, Gr, Sm}. In the latter situation, the geometric objects are often finite (e.g. various subgroup posets with natural $G$-actions), and the coefficients are the $G$-presheaves ($G$-equivariant contravariant functors over $G$-posets). In order to handle the coefficients, one may introduce a category $\mathbf{D}\rtimes G$ (see \cite{Gr}, written as $\mathbf{D}_G$ there), for a small $G$-category $\mathbf{D}$, and can define coefficients to be presheaves (contravariant fucntors) over this category. The aforementioned category is the algebraic predecessor of the Borel construction (or the homotopy orbit space) $B(\mathbf{D}\rtimes G)\simeq EG\times_G B\mathbf{D}$ \cite{DwHe} (or alternatively $[B\mathbf{D}/G].$ in \cite{De}). We shall call such a category a \textit{transporter category} over $\mathbf{D}$, which is a special \textit{Grothendieck construction} (a fibred category). 

In the present paper, we will mainly focus on the case where $\mathbf{D}=\calP$ is a finite $G$-poset, and investigate how the algebra of $\PG$ contributes to group representations. We shall show that the classical induction (Example 2.3), and the Harish-Chandra induction (Example 6.1), can be viewed as (co)homology with coefficients. Moreover even the Alvis-Curtis duality may be realized in a similar way. All of them are usually presented by tensor products with suitable bimodule complexes.

To set in a context, we briefly describe our general program. We shall call $\PG$ and its various quotient categories $\calC$, the \textit{local categories} of $G$ \cite{Xu1}. Several local categories are related by canonical functors
$$
\xymatrix{& \QH \ar[d] & \\
& \PG \ar[dr]^{\pi} \ar[dl]_{\rho} & \\
\calC & & G}
$$
We will focus on how to bridge up representations of $G$ and $\calC$, via $\PG$. The intrinsic properties of transporter categories and their representations are studied in \cite{Xu4}. 

Let $R$ be a commutative ring with identity. In the language of category algebras \cite{We, Xu1}, a $G$-presheaf $\mathfrak{F}$ of $R$-modules on $\calP$ is the same as a right $R\PG$-module. The above (co)homology constructions are certain higher limits $\lim^i_{\calP}\mathfrak{F}$ and $\colim^i_{\calP}\mathfrak{F}$, coming from the Kan extensions $LK_{\pi}\cong\lim_{\calP}, RK_{\pi}\cong\colim_{\calP}:$ mod-$R\PG\to$ mod-$RG$, classical constructions in category theory. It allows us to work with suitable (derived) module categories, in which all the above (co)homology groups have predecessors (finer invariants), and thus carry deeper information. In this way we assert that category representations contribute directly to group representations, see for example Theorem 3.1.

The paper is organized as follows. In Section 2, we review some basics about local categories, and point out the role of $\PG$ in various contexts. We prove, in Section 3, that each actual $RG$-module comes from an $R\PG$-module, under mild assumptions. Then in Sections 4, 5 and 6, we examine some $RG$-module complexes that afford interesting virtual modules. We begin with the analysis of relevant triangulated categories and functors, and end with a reformulation of the Alvis-Curtis duality. Section 7 addresses on an extension to certain small $G$-categories and its connection with Deligne-Lusztig induction.\\

\noindent\textbf{\sc Acknowledgements} I would like to thank my colleagues Ying Zong and Zhe Chen for stimulating discussions.

\section{Local categories}

The present paper may be considered as a continuation of \cite{Xu1}. Thus let us recall some fundamental constructions from there. Let $G$ be a finite group and $\calP$ be a small $G$-category (e.g. $\calS_p$, the poset of non-trivial $p$-subgroups and $\calS_p^1$, the poset of all $p$-subgroups of $G$). Let $x\in\Ob\calP$ and $\alpha\in\Mor\calP$. We shall denote by ${}^g x\in\Ob\calP$ and ${}^g\alpha\in\Mor\calP$ their images under the action by an element $g\in G$. We use 1 for the identity of $G$. Its action on $\calP$ is always trivial. 

If $\calC$ is a small category, we write $B\calC$ (the \textit{classifying space}) for the simplcial complex coming from the nerve of $\calC$ \cite{DwHe}. If $\calC$ admits a $G$-action, $B\calC$ becomes a $G$-simplcial complex (or a $G$-complex in short).

By definition, the \textit{transporter category} $\PG$, constructed on a $G$-category $\calP$, is a special kind of \textit{Grothendieck construction} and has the same objects as $\calP$. Meanwhile its morphisms are the formal products $\alpha g$, $\alpha\in\Mor\calP$ and $g\in G$. If $\beta h$ is another morphism that can be composed with $\alpha g$, then $(\alpha g)(\beta h):=(\alpha{}^g\beta)(gh)$. In fact, these two morphisms are composable in $\PG$ if and only if $\alpha$ and ${}^g\beta$ are composable in $\calP$. For instance, $\alpha g\in\Hom_{\PG}(x,y)$ itself is a composite $(\alpha 1)(1_{{}^gx}g)$, of $1_{{}^gx}g:x\to {}^g x$ and $\alpha 1: {}^g x\to y$. Thus $\alpha g$ may be understood as a ``conjugation'' followed by an ``inclusion''. (Be aware that in general $(1_{{}^gx}g)(\alpha 1)\ne \alpha g$. The former equals ${}^g\alpha g$.) As an example, one can check that a subgroup $H$ is identified with a skeleton of $(G/H)\rtimes G$, and thus they are equivalent. Here $G$ acts on $G/H$ by left multiplication.

From the definition, one easily constructs two canonical functors: an embedding
$$
\iota : \calP \to \PG
$$ 
(identity on objects and $\alpha\mapsto\alpha 1$ on morphisms) and 
$$
\pi : \PG \to G
$$ 
(obvious projection on objects and morphisms, induced by the $G$-functor $\calP \to pt$). 

A \textit{quotient category} $\calC$ of $\PG$ comes from a category extension \cite{We, Xu1}
$$
\calK \to \PG \to \calC.
$$
Here $\calK$ is by definition a functor from $\PG$ to the category of groups. Using axioms, it is easy to identify $\calK$ with a groupoid and a subcategory of $\PG$. Transporter categories of $G$ and their quotients are dubbed as \textit{local categories} of $G$.

\begin{remark} Transporter categories appear whenever there are group actions.
\begin{enumerate}
\item Given a $G$-space $X$, $EG\times_GX$ is sometimes called a homotopy orbit space, which often behaves better than the ordinary orbit space $X/G$ \cite{DwHe}, also see \cite{BeLu, De}. If $X\simeq B\calC$ for a small $G$-category $\calC$, then $EG\times_GX\simeq B(\calC\rtimes G)$.

\item In the theory of representation stability \cite{CEF, LY}, one considers the category FI whose objects are finite sets and morphisms are injective maps. Let $\mathbb{Z}_{> 0}=\{1,2,3,\cdots\}$ and $\calP$ be the power set, excluding the empty subset. The latter is indeed a poset where the relations are set inclusions. The infinite symmetric group $\Sigma_{\infty}=\lim_{\mathbb{N}}\Sigma_n$ is the group of \textit{finite} permutations on $\mathbb{Z}_{> 0}$. There is a natural $\Sigma_{\infty}$-action on $\calP$, while the resulting transporter category is exactly FI$=\calP\rtimes\Sigma_{\infty}$. In a similar fashion, we can construct the category VI.
\end{enumerate}
\end{remark}

Our base ring $R$ is commutative with identity, which often comes from three rings in a suitable $p$-modular system $(K,\calO,k)$ where char$k=p\bigm{|}|G|$. 

The $R$-representations of a small category $\calC$ are just (contravariant) functors on the category \cite{We}, which serve as coefficients when we compute (co)homology of $\calC$, and which form an abelian category, the functor category, $(\mbox{mod-}R)^{\calC}$. The following is a version of the original definition \cite{Be, CuLe, Sm}, and a more general construction is given by Grodal \cite{Gr}.

\begin{definition} Let $\calP$ be a $G$-category and $R$ be a commutative ring with identity. A \textit{$G$-equivariant $R$-coefficient system} on $\calP$ is the same as a contravariant functor $\PG \to$ mod-$R$. 

If $\Delta$ is a $G$-complex, a $G$-equivariant $R$-coefficient system on $\Delta$ is one on $\calP={\rm sd}\Delta$, the $G$-poset of simplices or the barycentric subdivision of $\Delta$.
\end{definition} 

We shall abbreviate ``$G$-equivariant $R$-coefficient systems'' to ``$G$-\textit{presheaves}''. 

If $\calC$ is a finite category, we shall investigate its representations through the category algebra $R\calC$ \cite{We, Xu3} since the category of finitely generated right $R\calC$-modules is equivalent to $(\mbox{mod-}R)^{\calC}$. If $\calC\simeq\calD$ as categories, then $R\calC$ is Morita equivalent to $R\calD$. For instance, given a subgroup $H$, $(G/H)\rtimes G$ is equivalent to $H$, and consequently $R[(G/H)\rtimes G]$ is Morita equivalent to $RH$. A $G$-presheaf on the $G$-poset $G/H$ is a contravariant functor $\mathfrak{N} : (G/H)\rtimes G \to$ mod-$R$ which assigns to each object a right $RH$-module $N$ (unique up to isomorphism) and to each morphism $1_{g_iH}g: g_jH\to g_iH$ (there is a unique $h\in H$ satisfying $gg_j=g_ih$) a map $\mathfrak{N}(1_{g_iH}g) : \mathfrak{N}(g_iH) \to \mathfrak{N}(g_jH)$ given by $\mathfrak{N}(1_{g_iH}g)(n)=nh$, $\forall n\in N$. Note that $\mathfrak{N}$ restricts to a presheaf on $G/H$, and is different from the constant presheaves that come from $RG$-modules, see \cite{Sm} and below. We will compute $\H_i(G/H;\mathfrak{N})$ later on in this section.

The category algebra $R\calC$ has a trivial module $\underline{R}$, determined by the constant functor from $\calC$ which sends every object to $R$ and every morphism to $\Id_R$. Recall that there exists an object-wise tensor product, written as $\hotimes$, between functors. It makes mod-$R\calC$ (and hence $D^b(\mbox{mod-}R\calC)$) into a symmetric monoidal category with tensor identity the trivial module $\underline{R}$. When $\calC=\PG$, for a finite $G$-poset $\calP$, there is a canonical isomorphism $R(\PG)\cong(R\calP)\rtimes G$ (the latter being the skew group algebra built on $R\calP$). Hence we shall write the transporter category algebra as $R\PG$. When $R$ is a field, this is a Gorenstein algebra, because both $R\calP$ and $RG$ are. It is the reason why we may consider the maximal Cohen-Macaulay modules \cite{Xu3}. For a Gorenstein algebra, a module is of finite projective dimension if and only if it is of finite injective dimension. If $\mathfrak{F}\in$ mod-$R\PG$, then $\mathfrak{F}$ is of finite projective dimension if and only if $\mathfrak{F}(x), \forall x$, is of finite projective dimension as an $RG_x$-module, where $G_x\subset G$ is the automorphism group of $x\in\Ob(\PG)$ (or equivalently the stabilizer of $x$) \cite{Xu1}.

The functor $\pi : \PG \to G$ induces a \textit{restriction} $Res_{\pi} : \mbox{mod-}RG \to \mbox{mod-}R\PG$ and we shall denote the values by $\kappa_M=Res_{\pi}M$, for each $M \in \mbox{mod-}RG$. We often call $\kappa_M$ a constant $G$-presheaf on $\calP$. When $M=R$, we have $\r=\kappa_R$. The restriction is equipped with two adjoint functors \cite{Xu1}
$$
LK_{\pi}, RK_{\pi} : \mbox{mod-}R\PG \to \mbox{mod-}RG,
$$
called the left and right \textit{Kan extensions}. Suppose $\mathfrak{F}\in\mbox{mod-}R\PG$. Then $\mathfrak{F}$ restricts to an $R\calP$-module, and
$$
LK_{\pi}\mathfrak{F}\cong\lim_{\calP}\mathfrak{F}\cong\mathfrak{F}\otimes_{R\calP}\r, RK_{\pi}\mathfrak{F}\cong\colim_{\calP}\mathfrak{F}\cong\Hom_{R\calP}(\r,\mathfrak{F}).
$$
We shall be mostly dealing with the left Kan extension. As for the right Kan extension, discussion is analogous. When $\calP=G/H$ for some subgroup $H$, the above functors are the classical restriction $\downarrow^G_H$, and induction $\uparrow^G_H$. 

For future reference, the derived functors of the Kan extensions are the higher limits, which are isomorphic to the category (co)homology \cite{We,  Xu1, Xu2} $\lim^i_{\calP}\mathfrak{F}=\H_i(\calP;\mathfrak{F})$ and $\colim^i_{\calP}\mathfrak{F}=\H^i(\calP;\mathfrak{F})$. By definition, the group $\H_i(\calP;\mathfrak{F})$ is the simplicial homology group of $\calP$ with coefficients in $\mathfrak{F}$. (The cohomology groups $\H^i(\calP;\mathfrak{F})$ are defined analogously.) In other words, they come from a complex of $RG$-modules $\sfC_*(\calP;\mathfrak{F})$ such that
$$
\sfC_i(\calP;\mathfrak{F})=\bigoplus_{x_0\to\cdots\to x_i}\mathfrak{F}(x_i). 
$$
Here $x_0, \cdots, x_i$ are objects of $\calP$, $x_0\to\cdots\to x_i$ is assumed to be a normalized $i$-chain (thus none morphism in the chain is an identity). (If $\calP$ is a finite $G$-poset, there are finitely many non-zero $\sfC_i(\calP;\mathfrak{F})$, each of which is of finite $R$-rank.) The differential is
$$
\hspace{-4cm}\partial(f_{x_0\to\cdots\to x_i})=\sum_{t=0}^{i-1}(-1)^t f_{x_0\to\cdots\to\hat{x_t}\to\cdots\to x_i} 
$$
$$
\hspace{5cm} + (-1)^i[\mathfrak{F}(x_{i-1}\to x_i)(f)]_{x_0\to\cdots\to x_{i-1}},
$$
where the subscript indicates to which summand an element belongs.

The $RG$-module structure on $\sfC_i(\calP;\mathfrak{F})$ is given as follows. For each $g\in G$, we have
$$
f_{x_0\to\cdots\to x_i}\cdot g = \{[\mathfrak{F}(1_{x_i}g)](f)\}_{{}^{g^{-1}}x_0\to\cdots\to {}^{g^{-1}}x_i}.
$$

By comparison, the set of normalized $i$-simplicies of $B\calP$ is a $G$-set, for each $i$, and it gives a right $RG$-module structure on $\sfC_i(B\calP,R)$ by $\sigma\cdot g:={}^{g^{-1}}\sigma$. Through identifying $\sfC_i(B\calP,R)$ with $\sfC_i(\calP;\r)$, we see the two actions coincide. One may look up for \cite[Chapter 10]{Sm} for further background material on homology representations.

\begin{example} Let $\mathfrak{N}$ be the $G$-presheaf on $G/H$ corresponding to an $RH$-module $N$. Then $\H_0(G/H;\mathfrak{N})=\bigoplus_i\mathfrak{N}(g_iH)$. If $gg_j=g_ih$, then $g$ acts on the summand $\mathfrak{N}(g_iH)=N$ via $\mathfrak{N}(1_{g_iH}g)$, that is $n\cdot g=nh \in \mathfrak{N}(g_jH)=N$ for every $n\in\mathfrak{N}(g_iH)$. Since $g_i^{-1}g=hg_j^{-1}$, $\H_0(G/H;\mathfrak{N})\cong N\otimes_{RH}RG$ is the usual induction of $N$. We emphasize that $\H_0(G/H;\mathfrak{N})\cong\H^0(G/H;\mathfrak{N})$ as $RG$-modules, since the classical induction is the same as the classical coinduction. (Then it will not be a surprise that later on one may use cohomology to define the Deligne-Lusztig induction.) 

Let $\H_*(\calP;\mathfrak{F})$ stand for the alternating sum $\sum_i(-1)^i\H_i(\calP;\mathfrak{F})$, a virtual module. Since $G/H$ is a (zero-dimensional) finite $G$-poset, we have $\H_*(G/H;\mathfrak{N})=\H_0(G/H;\mathfrak{N})$. 
\end{example}

To pass from a $G$-complex to a $G$-category, we record the following result which is perhaps well-known to the experts.

\begin{proposition} Let $\Delta$ be a $G$-complex and $\calP={\rm sd}\Delta$ be its barycentric subdivision. Then the canonical $G$-equivariant homeomorphism $\tau: B\calP \to \Delta$ induces a chain homotopy equivalence of $RG$-module complexes
$$
\sfC_*(\calP;\tau^*\mathfrak{F}) \to \sfC_*(\Delta;\mathfrak{F}),
$$
where $R$ is a commutative ring with identity and $\mathfrak{F}$ is a $G$-presheaf on $\Delta$. 
\end{proposition}

\begin{proof} The natural $G$-simplicial map (inducing a homeomorphism) \cite[Proposition 66.1]{CuRe} $\tau: B\calP \to \Delta$ maps a simplex $\sigma_0 \to\cdots\to \sigma_i$ to $\sigma_i$. A $G$-presheaf $\mathfrak{F}$ on $\Delta$ induces one on $B\calP$ via $\tau$, written as $\tau^*\mathfrak{F}$. It gives rise to a chain homotopy equivalence \cite[Chapter VI, Exercise 1]{Wh}, 
$$
\sfC_*(\calP;\tau^*\mathfrak{F}) \to \sfC_*(\Delta;\mathfrak{F}),
$$
such that $f_{\sigma_0\to\cdots\to\sigma_i}\in\sfC_i(\calP;\tau^*\mathfrak{F})=\bigoplus_{\sigma_0\to\cdots\to\sigma_i}\tau^*\mathfrak{F}(\sigma_i)$ is sent to $f_{\sigma_i} \in \sfC_i(\Delta;\mathfrak{F})=\bigoplus_{|\sigma|=i}\mathfrak{F}(\sigma)$. Note that if $|\sigma_i|\ne i$, then the corresponding image is set to be zero. 
\end{proof}

Finally, the functor $\iota: \calP\to\PG$ leads to a commutative diagram
$$
\xymatrix{\mbox{mod-}R\calP \ar[rrr]^{LK_{\pi}\ \ (\mbox{\tiny{resp.}}\ RK_{\pi})} & & & \mbox{mod-}R\\
\mbox{mod-}R\PG \ar[rrr]_{LK_{\pi}\ \ (\mbox{\tiny{resp.}}\ RK_{\pi})} \ar[u]^{Res_{\iota}} & & & \mbox{mod-}RG \ar[u]_{Res_{\iota}}}
$$

\section{Constructing actual group representations}

From Section 3 to Section 6, $G$-categories are finite posets and the resulting local categories are always finite. 

When we study the representations of $R\calC$, we may assume $\calC$ to be connected. Otherwise $R\calC$ will be a direct product of several $R\calC_i$, each $\calC_i$ being a connected component of $\calC$. However even if $\calC=\PG$ is connected, $\calP$ needs not to be. One may consider the example of $\calP=G/H$ for some proper subgroup $H$. We shall denote by $\Pi_0(\calP)=\Pi_0(B\calP)$ the number of connected components of $\calP$.

\begin{theorem} Given a connected transporter category $\PG$, if $\Pi_0(\calP)$ is invertible in $R$, then every $RG$-module $M$ is a direct summand of $LK_{\pi}\mathfrak{F}$, for some $R\PG$-module $\mathfrak{F}$.

Particularly if $\calP$ is connected, $M\cong LK_{\pi}\kappa_M$, for any $RG$-module $M$.
\end{theorem}

\begin{proof} Under the circumstance, $G$ permutes the set of connected components $\{\calP_1,\cdots,\calP_n\}$ of $\calP$, $n=\Pi_0(\calP)$. Set $H=\Stab_G(\calP_1)$. 

We know that $\H_*(\calP;\mathfrak{F})$ is computed by $\sfC_*(\calP;\mathfrak{F})$, and
$$
LK_{\pi}Res_{\pi}M\cong\lim_{\calP}Res_{\pi}M\cong\H_0(\calP;\kappa_M).
$$
Since $\sfC_*(\calP;\kappa_M)\cong\sfC_*(\calP;\r)\otimes_RM$, where $G$ acts diagonally, $-\otimes_RM$ is exact and $\sfC_*(\calP;\r)=\sfC_*(B\calP,R)$, we have $\H_0(\calP;\kappa_M)\cong\H_0(B\calP,R)\otimes_RM\cong R(G/H)\otimes_RM$. Note that the right $RG$-module structure on $R(G/H)$ is given by $g_iH\cdot g = g^{-1}g_iH$, and $M\otimes_RR(G/H)\cong M\otimes_{RH}RG$.
 
Consider the counit of the adjunction $\psi : LK_{\pi}Res_{\pi} \to \Id$. it admits a section $\eta : \Id \to LK_{\pi}Res_{\pi}$ as long as $n$ is invertible in $R$. It implies that each $RG$-module $M$ is a direct summand of $LK_{\pi}\mathfrak{F}$, with $\mathfrak{F}=\kappa_M$.
\end{proof}

In the language of monads, it leads to an equivalence between mod-$RG$ and some subcategory of mod-$R\PG$, under the assumption that $\Pi_0(\calP)$ is invertible in $R$. This will be discussed in a different place.

\begin{remark} Several interesting examples, in group representations, are consequences of this simple observation.
\begin{enumerate}
\item A special case is $\calP=G/H$, where $H$ contains a Sylow $p$-subgroup and $R$ is a field of characteristic $p$. 

\item Another well-exploited case is $\calP=\calS_p$, and $R$ is a field of characteristic $p$. Because Sylow $p$-subgroups distribute evenly in the connected components, $\Pi_0(\calS_p)$ is always coprime to $p$.

\item We may also consider the case of $\calS^1_p$, and $\calS_p$ if $O_p(G)\ne 1$, which are connected. (When $G$ is a finite group of Lie type in defining characteristic $p$, $\calS_p$ is always connected.) Under the circumstance we recover a result of Smith, see \cite[Section 10.1]{Sm}.
\end{enumerate}
\end{remark}

The second remark may be considered as the representation-theoretic counterpart of a topological statement that $BG\simeq_p B(\calS_p\rtimes G)$, which says that $\calS_p\rtimes G$ and $G$ has the same mod $p$ (co)homology. The idea is that one may glue up classifying spaces of certain subgroups to approximate that of $G$, see \cite[Chapter 7]{Sm}.

\section{On bounded derived categories and their Grothendieck rings}

We shall now turn to virtual modules. To this end, we will work with modules complexes. Derived module categories and their Grothendieck groups are the natural context. Our notations follow \cite[Section 6.8]{Zim}, but we write $D^b(A)=D^b(\mbox{mod-}A)$, for an algebra $A$ and its finitely generated right modules. The module category is identified with stalk complexes concentrated in degree 0.

Let $R$ be a field. It is known that $G_0(D^b(A))$ (the Grothendieck group of a triangulated category) is isomorphic to $G_0(\mbox{mod-}A)$ (the Grothendieck group of an Abelian category) \cite{Zim}. For our convenience, we shall denote both by $G_0(A)$, called the \textit{Grothendieck group} of $A$. For an object $\sfC_*\in D^b(A)$, one has in $G_0(A)$ that
$$
[\sfC_*]=\sum_i(-1)^i[\sfC_i]=\sum_i(-1)^i[\H_i(\sfC_*)].
$$
We are interested in the case of $A=R\calC$ for a local category $\calC$ of $G$. Under the circumstance, $G_0(R\calC)$ has a natural ring structure, and will be called the \textit{Grothendieck ring} of $R\calC$. At this point, let us focus on the transporter category $\PG$. Since $R\PG$ is Gorenstein \cite{Xu3}, we may continue to consider the stable category $\CM(R\PG)$, of maximal Cohen-Macaulay modules, which comes from a locolization sequence of tensor triangulated categories
$$
D^b(\mbox{proj-}R\PG)\to D^b(\mbox{mod-}R\PG)\to \CM(R\PG)
$$
and results in a short exact sequence of Abelian groups
$$
0\to G_0(D^b(\mbox{proj-}R\PG)){\buildrel{c}\over{\to}} G_0(R\PG)\to G^{st}_0(R\PG) \to 0.
$$
Here $G_0^{st}(R\PG)=G_0(\CM(R\PG))$ is the \textit{stable Grothendieck ring}, and $c$ is the Cartan map. The above maps are actually ring homomorphisms. Moreover since any module tensoring with a module of finite projective dimension is still of finite projective dimension, $G_0(D^b(\mbox{proj-}R\PG))$ becomes a (two-sided) ideal of $G_0(R\PG)$. 

The stable Grothendicek group is interesting only when the characteristic of $R$ divides the order of $G$, for otherwise $D^b(\mbox{proj-}R\PG)=D^b(\mbox{mod-}R\PG)$. As we mentioned earlier, if $\calP=G/G$ is trivial, then $R(G/G)\rtimes G\cong RG$. Under the circumstance, $\CM(RG)$ is commonly written as $\mbox{\underline{mod}-}RG$ or ${\rm stmod}RG$. 

Now we turn to compare bounded derived categories, based on the functor $\pi : \PG \to G$. The restriction $Res_{\pi}$ and the Kan extensions induce triangulated functors \cite{Xu3}
$$
\R_{\pi} : D^b(\mbox{mod-}RG) \to D^b(\mbox{mod-}R\PG)
$$
and
$$
\L_{\pi}, \mathbb{R}{\rm K}_{\pi} : D^b(\mbox{mod-}R\PG) \to D^b(\mbox{mod-}RG).
$$
\begin{remark}
Because $Res_{\pi}$ is exact, $\R_{\pi}$ is still adjoint to $\L_{\pi}$ and $\mathbb{R}{\rm K}_{\pi}$. 
\end{remark}

Since $\R_{\pi}$ maps $D^b(\mbox{proj-}RG)=D^b(\mbox{inj-}RG)$ into $D^b(\mbox{proj-}R\PG)=D^b(\mbox{inj-}R\PG)$, and $\L_{\pi}, \mathbb{R}{\rm K}_{\pi}$ does the converse, they produce functors between the stable categories
$$
\R'_{\pi} : \mbox{\underline{mod}-}RG \to \CM(R\PG)
$$
and
$$
\L'_{\pi}, \mathbb{R}{\rm K}'_{\pi} : \CM(R\PG) \to \mbox{\underline{mod}-}RG.
$$
Subsequently they give rise to maps between the (stable) Grothendieck groups
$$
r_{\pi}: G_0(RG)\to G_0(R\PG), 
$$

$$
r'_{\pi}: G^{st}_0(RG)\to G^{st}_0(R\PG),
$$

$$
lk_{\pi}, rk_{\pi}: G_0(R\PG)\to G_0(RG), 
$$

$$
lk'_{\pi}, rk'_{\pi}: G^{st}_0(R\PG) \to G^{st}_0(RG).
$$
Taking the tensor structures into account, it is obvious that both $r_{\pi}$ and $r'_{\pi}$ are ring homomorphisms. We may explicitly write out the map $lk_{\pi}$ (and analogously $lk'_{\pi}, rk_{\pi}, rk'_{\pi}$) by
$$
[\mathfrak{F}]\mapsto\sum_i(-1)^i[\H_i(\calP;\mathfrak{F})],
$$
for each $\mathfrak{F}\in \mbox{mod-}R\PG$. It makes sense since $\calP$ is finite and the above becomes a finite alternating sum.

\section{On homology representations} 

For an arbitrary finite group $G$, suppose $\calP=\calB_p$ is the $G$-poset of $p$-subgroups $V$ satisfying $V=O_p(N_G(V))$ (Bouc's poset, see \cite{Sm}). It is always $G$-homotopy equivalent to $\calS_p$.

Let $G$ be a finite group with a split BN-pair, of rank $|S|>1$ and characteristic $p$ (here $(W,S)$ stands for the Weyl group and its set of distinguished generators), see \cite{CuRe} for background. The Tits building $\Delta$ is a simplicial complex whose simplices are indexed by the parabolic subgroups, and where the inclusion relation is opposite to the inclusions of parabolic subgroups. One often needs to consider the (co)homology of $\Delta$ with coefficients, see for instance \cite{CuRe, Sm}. 

\begin{example} Let $G$ be a finite group with a split BN-pair, of rank $|S|>1$ and characteristic $p$. Since $\calB_p$ is the poset of unipotent radicals (of parabolic subgroups), it is identified with the poset of simplices in $\Delta$. It means that $\calB_p\cong\sd\Delta$ is the barycentric subdivision of $\Delta$ and we may apply Proposition 2.4., where $\tau: B\calB_p \to \Delta$ maps a simplex $V_{I_0}\to\cdots\to V_{I_i}$ to $P_{I_i}$.

If $\underline{R}$ is the constant coefficient system on $\Delta$, $\tau^*\underline{R}=\underline{R}$ and there are chain homotopy equivalence of $RG$-module complexes
$$
\sfC_*(\calB_p;\underline{R})\simeq\sfC_*(\Delta;\underline{R}).
$$
The augmented complex gives rise to the \textit{Steinberg module} (differ by a sign $(-1)^{|S|-1}$)
$$
\st_G:=(-1)^{|S|-1}\sum_{I\subseteq S}(-1)^{|I|}[R\uparrow^G_{P_I}],
$$
known to be afforded by $\H_{|S|-1}(\Delta,R)=\H_{|S|-1}(\Delta;\r)$ \cite{CuRe}. Here $R$ can be either $K$ or $k$ from a $p$-modular system $(K,\calO,k)$.
\end{example}

Let $G$ be an arbitrary finite group. In \cite{Xu2}, we had a formula that, for an $RG$-module $M$, regarded as a stalk complex at degree 0,
$$
\L_{\pi}\R_{\pi}M\cong\L_{\pi}\kappa_M\cong\sfC_*(\calP;\kappa_M)\cong\sfC_*(\calP;\underline{R})\otimes_R M,
$$
where $G$ acts diagonally on the tensor product. It is based on an explicit calculation that the bar resolution of $\r$ is sent to a projective resolution of $\sfC_*(\calP;\r)$, by the left Kan extension. We readily deduce that (if $R$ is a field)
$$
lk_{\pi}([\kappa_M])=[\sfC_*(\calP;\kappa_M)]=\sum_{i}(-1)^i[\H_i(\calP;\kappa_M)]\in G_0(RG),
$$ 
which matches our definition of $lk_{\pi}$ in last section. It prompts us to look further into some module complexes.

Since there is a counit $\Psi: \L_{\pi}\R_{\pi}\to \Id_{D^b(\mbox{\tiny{mod-}}RG)}$, it results in, for each $RG$-module $M$, a map $\Psi_M : \L_{\pi}\R_{\pi}M \to M$. Combining previous calculations, it is identified with the augmentation
$$
\Psi_M: \sfC_*(\calP;\kappa_M) \to M,
$$
which on degree zero is the summation $\sfC_0(\calP;\kappa_M)=\bigoplus_{x\in\Ob\calP}M \to M$. 

\begin{proposition} The mapping cone $Cone(\Psi_M)$ represents the following element in $G_0(RG)$
$$
\sum_{i}(-1)^i[\sfC_i(\calP;\kappa_M)]-[M].
$$
If $R$ is a field, it is identified with $\sum_{i}(-1)^i[\H_i(\calP;\kappa_M)]-[M]$.
\end{proposition}

When $\calP=\calS_p$ and $M=R$, $Cone(\Psi_M)$ may be called the \textit{generalized Steinberg module}. If moreover $R=k$ is a field of characteristic $p$, $Cone(\Psi_k)$ (the generalized Steinberg module at $p$) is proved by Webb to be an object of $D^b(\mbox{proj-}kG)$, see \cite{Sm}. Be aware that in other places, homology representations and the (generalized) Steinberg representation are constructed in $a(RG)$, the representation ring (see for instance \cite{CuRe}). Here we study them in $G_0(RG)$, through the canonical surjection $a(RG)\to G_0(RG)$.

\begin{definition}
Let $R$ be a commutative ring with identity. We call $\H_{\calP}(G;\mathfrak{F})=\sum_{i}(-1)^i[\H_i(\calP;\mathfrak{F})]\in G_0(RG)$ the \textit{homology representation} of $G$ on $\calP$ with coefficients in $\mathfrak{F}\in\mbox{mod-}R\PG$, and $\st_{\calP}(G)=\H_{\calP}(G;\r)-[R]$ the \textit{generalized Steinberg representation} of $G$ on $\calP$. 

We shall abbreviate $\H_{\calP}(G;\r)$ as $\H_{\calP}(G)$.
\end{definition}

Note that in \cite{CuRe}, $\bigoplus_{i}\H_i(\calP;\r)\in \mbox{mod-}RG$ is called the homology representation of $G$ on $\calP$. However, it seems more natural to go with the signs, from what we have seen. Proposition 5.2 may be rephrased as follows.

\begin{corollary} Let $G$ be a finite group, $p$ be a prime that divides $|G|$, $\calP$ be a finite $G$-poset, $R$ be a field, and $M$ be a right $RG$-module. Then
$$
lk_{\pi}([\kappa_M])=[\sfC_*(\calP;\r)][M]=\H_{\calP}(G)[M]\in G_0(RG).
$$
It follows that 
$$
lk_{\pi}r_{\pi}([M])=lk_{\pi}([\kappa_M])=[M]+\st_{\calP}(G)[M].
$$
or
$$
(lk_{\pi}r_{\pi}-\Id)([M])=\st_{\calP}(G)[M].
$$
Particularly, when $\calP=\calS_p$, the operator $lk_{\pi}r_{\pi}-\Id : G_0(RG) \to G_0(RG)$ maps $[R]$ to $\st_p(G)$, the generalized Steinberg module at $p$.
\end{corollary}

It rings a bell when one compares the last statement with a property of the Alvis-Curtis duality, a self-map on $G_0(\CC G)$ that exchanges $[\CC]$ with $\st_G$ when $G$ is a finite group of Lie type, see Theorem 6.2.

We record the following observations.

\begin{corollary} Let $\calP$ be a finite $G$-poset.
\begin{enumerate}
\item If $\calP$ is connected and has vanishing homology, then both $r_{\pi}$ and $r'_{\pi}$ are injective. Meanwhile both $lk_{\pi}$ and $lk'_{\pi}$ are surjective.

\item If $\st_{\calP}(G)$ is virtual projective, then $r'_{\pi}$ is injective. Meanwhile $lk'_{\pi}$ is surjective.
\end{enumerate}
\end{corollary}

It will be interesting to compute $lk_{\pi}([S])$ for each simple $R\PG$-module. Regarded as a functor, $S$ must be atomic, in the sense there exists an object $x\in\Ob(\PG)$ such that $S_{x,V}(y)=0$ for all $x\not\cong y$ in $\Ob(\PG)$. Moreover, as a functor we must have $S(z)\cong S(x)$, if $z\cong x$, as $RG_x$-modules. Here $G_x=\Aut_{\PG}(x)$ is the stabilizer of $x$ in $G$. Using techniques developed by Bouc, Oliver, Quillen, Thevenaz, Webb et al \cite{Sm}, we deduce that
$$
\begin{array}{ll}
lk_{\pi}([S])&=\sum_i(-1)^i[\H_i(\calP;S)]\\
&\\
&=\sum_i(-1)^i[\sum_{y\cong x}\H_i(\calP_{\ge y},\calP_{> y};S(y))]\\
&\\
&=\sum_i(-1)^i[\H_i(\calP_{\ge x},\calP_{> x};S(x))]\uparrow^G_{G_x}\\
&\\
&=\sum_i(-1)^i([\H_i(\calP_{\ge x},\calP_{> x};k)][S(x)])\uparrow^G_{G_x}\\
&\\
&=\sum_i(-1)^i([\tilde{\H}_{i-1}(\calP_{> x};k)][S(x)])\uparrow^G_{G_x}.
\end{array}
$$

The last equality is true because $\calP_{\ge x}$ has an initial object and hence vanishing homology. Note that if $x$ is maximal, then $\calP_{>x}=\emptyset$, which has $k$ as its -1 degree reduced homology and zero elsewhere. It means when $x$ is maximal, $lk_{\pi}([S])=[S(x)]\uparrow^G_{G_x}$.

Since for each simple module $S$, we always have $lk_{\pi}([S])\in G_0(RG_x)\uparrow^G_{G_x}$ for some $x$, when $lk_{\pi}$ (or $lk'_{\pi}$ if applicable) is surjective, we get the following ``induction theorems''. 

\begin{corollary} Let $\calP$ be a finite $G$-poset.
\begin{enumerate}
\item If $\calP$ is connected and has vanishing homology, then $$
G_0(RG)=\sum_{[x]\subset\Ob(\PG)}G_0(RG_x)\uparrow^G_{G_x}.
$$
and
$$
G_0^{st}(RG)=\sum_{[x]\subset\Ob(\PG)}G_0^{st}(RG_x)\uparrow^G_{G_x}.
$$

\item If $\st_{\calP}(G)$ is virtual projective, then 
$$
G_0^{st}(RG)=\sum_{[x]\subset\Ob(\PG)}G_0^{st}(RG_x)\uparrow^G_{G_x}.
$$
\end{enumerate}
\end{corollary}

\section{Orbit categories and the Alvis-Curtis duality}

Motivated by Section 5, we try to give a new construction of the Alvis-Curtis duality, using homology with local coefficients. Throughout this section, we suppose $G$ is a finite group with a split BN-pair, of rank $|S|>1$ and characteristic $p$. We do need the orbit category $\calO_{\calB_p}(G)$ to provide an appropriate coefficient system, introduced by Curtis \cite{Cu}. The point is that the automorphism groups in the orbit category are exactly the Levi subgroups of $G$. In order to proceed, we will see that both the left and right Kan extensions need to be involved.

There are the Harish-Chandra restriction and induction relating $RG$-mod with $RL$-mod, for $L$ a Levi subgroup. It may be pictured as the lower left diagram.
$$
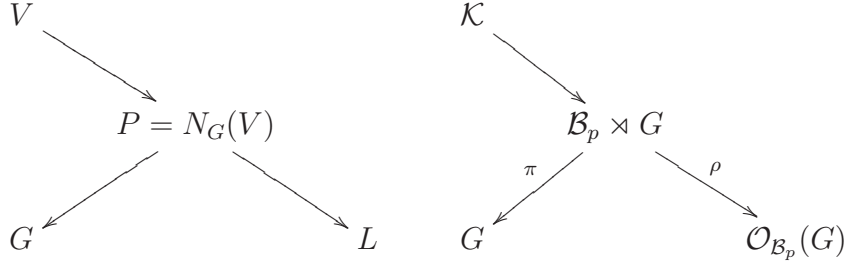
\xymatrix{V \ar[dr]&&&\calK \ar[dr]&&\\
&P=N_G(V)\ar[dr]\ar[dl]&&&\calB_p\rtimes G\ar[dr]^{\rho} \ar[dl]_{\pi}&\\
G&&L&G&&\calO_{\calB_p}(G)}
$$

\begin{example} If we fix some standard parabolic subgroup $P_I$ with Levi decomposition $L_I\ltimes V_I$, then $\calP_I=\{{}^gV_I\bigm{|} g\in G\}$ is a $G$-subposet of $\calB_p$. (If one wants to be consistent, $\calP_I$ is isomorphic to $\{gP_I\bigm{|} g\in G\}$ with left multiplications.) We see that $\calP_I\rtimes G\simeq P_I$ and the corresponding orbit category $\calO_I$  is equivalent to $L_I$, with a natural functor $\calP_I\rtimes G \to \calO_I$. In this way, any $RL_I$-module $N$ determines a $G$-presheaf $\mathfrak{N}$ on $\calO_I$. It inflates/restricts to a $G$-presheaf $\mathfrak{N}$ on $\calP_I$. One can compute directly and show that $\H_*(\calP_I;\mathfrak{N})=\H_0(\calP_I;\mathfrak{N})$ is exactly the Harish-Chandra induction of $N$. Again as in the classical case, one can use cohomology $\H^0(\calP_I;\mathfrak{N})\cong\H_0(\calP_I;\mathfrak{N})$ to define the Harish-Chandra induction. 
\end{example}

Replacing the group extension by a category extension, the diagram on the right becomes a categorical version of the left whose automorphism groups are exactly the Levi subgroups of $G$. We wish to lift the previously mentioned Harish-Chandra restriction and induction to functors relating representations of $G$ and $\calO_{\calB_p}(G)$, as an ``amalgam'' of Levi subgroups. The functor $\calK$ maps each object $V\in\Ob(\calB_p\rtimes G)$ to $V$ itself, and can be identified with the groupoid consisting of all unipotent radicals. In this case however the homology of $\calB_p$ does not vanish at all positive dimensions.

From the canonical functors $G {\buildrel{\pi}\over{\leftarrow}} \calB_p\rtimes G {\buildrel{\rho}\over{\to}} \calO_{\calB_p}(G)$, we obtain functors between module categories
$$
{\rm T}_G=RK_{\rho}Res_{\pi} : \mbox{mod-}RG {\buildrel{Res_{\pi}}\over{\longrightarrow}} \mbox{mod-}R\calB_p\rtimes G {\buildrel{RK_{\rho}}\over{\longrightarrow}} \mbox{mod-}R\calO_{\calB_p}(G),
$$
and
$$
{\rm I}_G=LK_{\pi}Res_{\rho} : \mbox{mod-}RG {\buildrel{LK_{\pi}}\over{\longleftarrow}} \mbox{mod-}R\calB_p\rtimes G {\buildrel{Res_{\rho}}\over{\longleftarrow}} \mbox{mod-}R\calO_{\calB_p}(G).
$$
By the adjunctions between the restriction and the Kan extensions, we readily verify that 
$$
\Hom_{RG}({\rm I}_G(\mathfrak{F}),N)\cong\Hom_{R\calO_{\calB_p}}(\mathfrak{F},{\rm T}_G(N))
$$
which means that ${\rm I}_G$, called the \textit{integrated Harish-Chandra induction}, is the left adjoint of ${\rm T}_G$, called the \textit{integrated Harish-Chandra restriction}. 

Let us exploit the basic properties of these two new functors that we have just introduced. To understand ${\rm T}_G$, we must calculate $RK_{\rho}$. In fact, as we have seen in \cite{Xu1}, since $\calB_p\rtimes G \to \calO_{\calB_p}(G)$ is part of a category extension, for every $\mathfrak{F}\in$ mod-$R\calB_p\rtimes G$,
$$
RK_{\rho}(\mathfrak{F})=\colim_{\calK}\mathfrak{F}\cong\H^0(\calK;\mathfrak{F}).
$$
At an object $V\in\Ob\calB_p$, $RK_{\rho}(\mathfrak{F})(V)=\H^0(V,\mathfrak{F}(V))=\mathfrak{F}(V)^V$. If there is a morphism $V \to V'$, one easily defines $\mathfrak{F}(V')^{V'} \to \mathfrak{F}(V)^{V}$. 

In the rest of this section, we shall be mainly working over $R=\CC$, but many constructions are valid over a ring $R$ in which $p$ is invertible. Given a finite EI category $\calC$ \cite{We}, every right $\CC\calC$-module is of finite projective dimension. Thus
$$
G_0(\CC\calC)=G_0(D^b(\mbox{mod-}\CC\calC))=G_0(D^b(\mbox{proj-}\CC\calC))\cong  K_0(\CC\calC).
$$
For consistency, we shall stick with the notation $G_0(\CC\calC)$.

Under the circumstance, the functor $RK_{\rho}$ is exact, and hence so is ${\rm T}_G$. It follows that ${\rm I}_G$ is right exact. The induced map $rk_{\rho} : G_0(\CC\calB_p\rtimes G) \to G_0(\CC\calO_{\calB_p}(G))$ is naturally given by 
$$
[\frakN] \mapsto [RK_{\rho}\frakN].
$$

The two functors ${\rm I}_G$ and ${\rm T}_G$ induce derived functors $\I_G$ and $\T_G$. Since ${\rm T}_G$ is exact, the two functors $\I_G$ and $\T_G$ are adjoint to each other with a counit $\Theta : \I_G\T_G \to \Id_{D^b(\CC\calO_{\calB_p})}$. 

The two functors also induce maps between the Grothendieck groups
$$
i_G : G_0(\calO_{\calB_p}(G)) \to G_0(\CC G);\ \ \  [\mathfrak{F}] \mapsto \sum_i(-1)^i[\H_i(\calB_p;\mathfrak{F})],
$$
and
$$
t_G : G_0(\CC G) \to G_0(\CC\calO_{\calB_p}(G));\ \ \ [M] \mapsto [\H^0(\calK;\kappa_M)].
$$

Cabanes-Rickard \cite{CaRi} proved that the Alvis-Curtis duality is the consequence of an existing category self-equivalence on $D^b(\CC G)$, and Okuyama \cite{Ca} later on lifted the equivalence to the homotopy category, following a conjecture of Cabanes-Rickard. We want to show that the counit $\Theta$ leads to their construction. To this end, we recall that they introduced a $\CC G$-bimodule complex $X$ such that
$$
X^i=\bigoplus_{|I|=|S|-i-1}\CC G\bar{V_I}\otimes_{\CC P_I}\bar{V_I}\CC G,
$$
Here $\bar{V_I}$ is the sum of elements of $V_I$, multiplied by $|V_I|^{-1}$. Especially $X^{|S|}=\CC G$ is set at degree $-1$. Cabanes and Rickard introduced an equivalence 
$$
-\otimes_{\CC G}X : D^b(\CC G)\to D^b(\CC G),
$$
which on a right $\CC G$-module $M$, at degree $i$, is
$$
M\otimes_{\CC G}X^i=\bigoplus_{|I|=|S|-i-1}M\bar{V_I}\otimes_{\CC P_I}\bar{V_I}\CC G\cong \bigoplus_{|I|=|S|-i-1}M\bar{V_I}\otimes_{\CC L_I}\bar{V_I}\CC G.
$$
Be aware that, Cabanes-Rickard worked on an coefficient ring where $p$ is invertible, and considered \textit{left} modules. They showed that $M\otimes_{\CC G}X$ is isomorphic to $\widetilde{\sfC}_*(\Delta;\mathfrak{F}_M)$ (the augmented chain complex), where the $G$-presheaf $\mathfrak{F}_M$ on $\Delta$ dedined by $\mathfrak{F}_M(P_I)=M^{V_I}$ \cite{Cu, Sm}.

\begin{theorem} Suppose $G$ is a finite group with split BN pair, of type $(W,S)$ such that $|S|>1$. Then
$$
Cone(\Theta_-)\cong -\otimes_{\CC G} X,
$$
and consequently
$$
D_G:=i_Gt_G-\Id : G_0(\CC G) \to G_0(\CC G)
$$
is the Alvis-Curtis duality, up to a sign.
\end{theorem}

\begin{proof}  We note that the $G$-presheaf $\H_0(\calK;\kappa_M)$ on $\calB_p$ is exactly $\tau^*\mathfrak{F}_M$. By Proposition 2.2, for each $\CC G$-module $M$, $\Theta_M : \I_G\T_G(M) \to M$ is identified with the augmented chain complex
$$
\sfC_*(\calB_p;\mathfrak{F}_M) \to M,
$$
and $Cone(\Theta_M)\cong Cone(\sfC_*(\Delta;\mathfrak{F}_M) \to M)=\widetilde{\sfC}_*(\Delta;\mathfrak{F}_M)$, which is isomorphic to $M\otimes_{\CC G}X$.

The second statement actually follows from the above, as well as Cabanes-Rickard. We can also establish it by explicitly computation, as in Example 5.1.
\end{proof}

We are unable to provide an intrinsic proof that $Cone(\Theta_-)$ is a category equivalence, although we suspect that it could come in one way or another from the adjunction between $\I_G$ and $\T_G$.

\begin{remark} The quasi-inverse of $-\otimes_{\CC G}X$ is given by $-\otimes_{\CC G}X^{\vee}$, where $X^{\vee}$ is the linear dual of the $\CC G$-bimodule complex. Let $M^{\circ}$ denote the contragredient module of $M$. Following the formula \cite[Corollary 2.2]{CaRi} that 
$$
M\otimes_{\CC G}X^{\vee} \cong (M^{\circ}\otimes_{\CC G}X)^{\circ},
$$
we see that $M\otimes_{\CC G}X^{\vee}\cong\widetilde{\sfC}_*(\calB_p;\mathfrak{F}_{M^{\circ}})^{\circ}\cong\widetilde{\sfC}^*(\calB_p;\mathfrak{F}_{M^{\circ}})$. It is interesting to observe that $-\otimes_{\CC G}X^{\vee}$ is realized by the composite the the lower three functors in the diagram
$$
\xymatrix{D^b(\CC G) \ar[rr]^{-\otimes_{\CC G} X} & & D^b(\CC G) \ar[d]^{(-)^{\circ}} \\
D^b(\CC G) \ar[u]^{(-)^{\circ}} & & D^b(\CC G) \ar[ll]^{-\otimes_{\CC G} X}}
$$
Since the horizontal functors are covariant and the vertical ones are contravariant, the contravariant functor 
$$
(-\otimes_{\CC G}X)^{\circ}\cong Cone(\Theta_-)^{\circ}\cong\widetilde{\sfC}^*(\calB_p;\mathfrak{F}_-)
$$
deserves to be called the Alvis-Curtis duality similar to the comments in \cite[Section 7.1]{CaRi}. Along with the cohomological constructions of the classical and Harish-Chandra inductions, it hints that cohomology may be more reasonable than homology to formulate the inductions and the Alvis-Curtis duality.
\end{remark}

\section{On Deligne-Lusztig induction}

To be consistent, we speculate on the case of Deligne-Lusztig induction. In this situation, one has to consider infinite but small $G$-categories, as well as the $l$-adic $G$-equivariant sheaves. A $l$-adic sheaf is a limit of constructible sheaves which requires further considerations and which we do not discuss here. Due to the different nature of presheaves and sheaves, only Kan extensions are not sufficient to handle the situation.

Let $X$ be a scheme, separated and of finite type over an algebraically closed field, say $\overline{\mathbb{F}_q}$, $q=p^m$, and $X_{\rm et}$ to be the small site on it \cite[Expos\'e VII]{AGV}, whose underlying category, still written as $X_{\rm et}$, is small. If furthermore $X$ admits a $G$-action, $X_{\rm et}$ is a $G$-category. Suppose $\mathfrak{F}$ is a $G$-equivariant sheaf of $R$-modules on $X_{\rm et}$. The resulting \'etale cohomology groups $\H^i(X;\mathfrak{F}):=\H^i(X_{\rm et};\mathfrak{F})$, $ i\ge 0$, are $RG$-modules.

Suppose $G$ is a finite group of Lie type, in defining characteristic $p$ and $X$ is a Deligne-Lusztig variety \cite{DeLu, Sr} that is acted on by $G$ and a Levi subgroup $L$. Cohomology with compact support $\H^*_c(X;\mathfrak{F})$ is needed. By a theorem of Nagata, there is a ``compactification'' $\widetilde{X}$ of $X$, a scheme proper over $k$, along with an open immersion $j: X \to \widetilde{X}$. Let $\mathfrak{F}$ be a torsion sheaf, there is an extension by zero $j_!\mathfrak{F}$ on $\widetilde{X}_{\rm et}$, and $\H^i_c(X;\mathfrak{F}):=\H^i(\widetilde{X};j_!\mathfrak{F})$. If $R$ (e.g. $\mathbb{Z}/l^n\mathbb{Z}$, prime $l\ne p$) is a finite commutative ring and $\mathfrak{F}$ is constructible, then $\H^i_c(X;\mathfrak{F})$ is finite for each $i$ and vanishes at $i > 2d$. Using Poincar\'e duality \cite{Mi}, as $RG$-modules $\H^i_c(X;\mathfrak{F})\cong\Hom_R(\H^{2d-i}(X;\check{\mathfrak{F}}(d)),R)$, where $d$ is the dimension of $X$.

Let $\calC$ be a small site with $G$-action. The $G$-equivariant sheaves on $\calC$ form an Abelian category ${\rm Sh}_G(\calC)$,  having enough injectives (\cite{BeLu, De}). The category mod-$R\calC\rtimes G$ now stands for the category of $G$-equivariant presheaves/functors on $\calC$, and many constructions in Section 4 still work. Note that the forgetful functor $i : {\rm Sh}_G(\calC)\hookrightarrow\mbox{mod-}R\calC\rtimes G$ is only \textit{left exact}. In general we have
$$
\xymatrix{{\rm Sh}(\calC) \ar[r]^{i} & \mbox{mod-}\calC \ar[rr]^{RK_{\pi}} & & \mbox{mod-}R\\
{\rm Sh}_G(\calC) \ar[r]_{i} \ar[u]^{Res_{\iota}} & \mbox{mod-}R\calC\rtimes G \ar[rr]_{RK_{\pi}} \ar[u]^{Res_{\iota}} & & \mbox{mod-}RG={\rm Sh}_G(pt) \ar[u]_{Res_{\iota}}}
$$
The vertical (forgetful) functors are exact, and they have adjoint functors (induction=co-induction). The adjoint functors of $RK_{\pi}$ and $i$, $Res_{\pi}$ and the shifification $\sharp$, are exact. Note that the sheafification of a $G$-equivariant presheaf is always a $G$-equivariant sheaf.

Since $G$ is \textit{discrete}, the equivariant derived category $D^b_G(\calC)$ is simply defined as $D^b({\rm Sh}_G(\calC))$ \cite{BeLu}. Denote by $D^b_{G,c}(\calC)$ the (thick) triangulated subcategory of $D^b_G(\calC)$ that consists of constructible complexes. One has $D^b_{G,c}(pt)=D^b_G(pt)=D^b(RG)$.

Let $\calC=X_{\rm et}$ where $X$ is a Deligne-Lusztig variety as above. Since $\Id_X: X \to X$ is a final object in $X_{\rm et}$, $RK_{\pi}=\colim_{X_{\rm et}}=\Hom_{X_{\rm et}}(\r,-)$ is exact and the global section functor 
$$
\Gamma_{X}=RK_{\pi}\circ i: {\rm Sh}_G(X_{\rm et}) \hookrightarrow \mbox{mod-}R X_{\rm et}\rtimes G \to \mbox{mod-}R
$$ 
is identified with $\Hom_{X_{\rm et}}(\r,-)$ for $\Gamma(\mathfrak{G})=\mathfrak{G}(\Id_X)=\Hom_{X_{\rm et}}(\r,\mathfrak{G})$, where $\r$ is the constant presheaf/functor on $X_{\rm et}$ and $\mathfrak{G}\in{\rm Sh}_G(X_{\rm et})$. Thus given a torsion sheaf $\mathfrak{F}\in{\rm Sh}_G(X_{\rm et})$, $\H^{2d-i}(X;\check{\mathfrak{F}}(d))$ is computed by an injective resolution $\mathfrak{I}_*$ of $\check{\mathfrak{F}}(d)\in{\rm Sh}_G(X_{\rm et})$, such that
$$
\Hom_{{\rm Sh}(X_{\rm et})}(\r^{\sharp},\mathfrak{I}_*)\cong\Hom_{X_{\rm et}}(\r,i(\mathfrak{I}_*))\cong i(\mathfrak{I}_*)(\Id_X).
$$
Since $G$ is contained in the stabilizer of $\Id_X\in\Ob X_{\rm et}$, the complex is obviously an $RG$-complex. Moreover, since each $i(\mathfrak{I}_t), t\ge 0,$ is an injective presheaf/functor over $X_{\rm et}\rtimes G$, it has to be injective on evaluation at every object $\phi : Y\to X$, as an $RG_{\phi}$-module. However a group algebra is self-injective. It means that $i(\mathfrak{I}_t)(\phi)$ is always a projective $RG_{\phi}$-module, and particularly $i(\mathfrak{I}_*)(\Id_X)$ is a complex of projective $RG$-modules. If $R=\mathbb{Z}/l^n\mathbb{Z}$ and $\mathfrak{F}$ is constructible, $i(\mathfrak{I}_*)(\Id_X)$ is isomorphic to a perfect complex in $D^b(RG)$, see for instance \cite[P.68]{Sr}. 

\begin{theorem} Let $G$ be a finite group of Lie type in defining characteristic $p$, $X$ be a Deligne-Lusztig variety and $R=\mathbb{Z}/l^n\mathbb{Z}$ for $l$ a prime different from $p$. Then the right Kan extension along $\pi : X_{\rm et}\rtimes G \to G$ induced a functor
$$
\mathbb{R}^i\Gamma_X : D^b_{G,c}(X_{\rm et}) \to D^b(RG),
$$
which gives rise to a map
$$
\gamma_X: G_0(D^b_{G,c}(X_{\rm et})) \to G_0(D^b(RG))
$$ 
sending $\r^{\sharp}$ to $\H^*(X;\check{\r^{\sharp}}(d))=\Hom_R(\H^*_c(X;\r^{\sharp}),R)$. Here $\r$ is the trivial $RX_{\rm et}\rtimes G$-module.
\end{theorem}

\begin{proof} The adjoint pair $\Gamma_X=RK_{\pi}\circ i$ and $\sharp\circ Res_{\pi}$ (exact) induce adjoint functors between $D^b_G(X_{\rm et})$ and $D^b_G(pt)$, which restrict to functors between the triangulated subcategories of constructible complexes $D^b_{G,c}(X_{\rm et})$ and $D^b_{G,c}(pt)=D^b(RG)$.
\end{proof}

\begin{remark} Ying Zong told me a way to construct a $G$-equivariant $\widetilde{X}$, when $G$ is finite, so that one may produce $\H^*_c(X;\r^{\sharp})$ directly (considering $G_0(D^b_{G,c}(\widetilde{X}_{\rm et}))$).
\end{remark}

Let $T$ be a maximal torus and $\theta$ a linear character. There is a $l$-adic sheaf $\mathfrak{F}_{\theta}$ so that the Deligne-Lusztig induction is $R^G_T(\theta)=\H^*_c(X;\mathfrak{F}_{\theta})$. However, we do not know how to establish a similar result for $\mathbb{Z}_l$-sheaves.

\end{document}